\documentclass[english,a4paper,12pt,twoside,reqno]{amsart}

\usepackage{amsmath,amsfonts,latexsym,amssymb,amsthm,amscd}
\usepackage{xcolor}
\usepackage[tips,matrix,arrow,frame]{xy}
\usepackage{tikz}
\usepackage{etex}

\numberwithin{equation}{section}

\newcommand{\C}{\mathbb{C}}

\newcommand{\Z}{\mathbb{Z}}

\newcommand{\fgl}{\mathfrak{gl}}

\newcommand{\diag}{\mathrm{diag}}
\newcommand{\End}{\mathrm{End}}
\newcommand{\Ker}{\mathrm{Ker}}

\newcommand{\id}{\mathrm{id}}

\newcommand{\off}{\mathrm{off}}
\newcommand{\op}{\mathrm{op}}

\newcommand{\Tr}{\mathrm{Tr}}

\newtheorem{thm}{Theorem}[section]
\newtheorem{prop}[thm]{Proposition}
\newtheorem{lemma}[thm]{Lemma}
\newtheorem{cor}[thm]{Corollary}

\theoremstyle{definition}
\newtheorem{rem}[thm]{Remark}
\newtheorem{ack}{Acknowledgment}
\newtheorem{free text}[thm]{}

\newcommand{\fg}{\mathfrak{g}}

\newcommand{\fm}{\mathfrak{m}}

\newcommand{\cB}{\mathcal{B}}

\newcommand{\cR}{\mathcal{R}}

\begin{document}
\title{Homology of the Lie algebra $\fgl(\infty,R)$ }
\author{A. Fialowski and K. Iohara}
\address{University of P\'ecs and E\"{o}tv\"{o}s Lor\'{a}nd University Budapest,  Hungary}
\email{fialowsk@ttk.pte.hu, fialowsk@cs.elte.hu}
\address{Univ Lyon, Universit\'{e} Claude Bernard Lyon 1, CNRS UMR 5208, Institut Camille Jordan, 
43 Boulevard du 11 Novembre 1918, F-69622 Villeurbanne cedex, France}
\email{iohara@math.univ-lyon1.fr}

\begin{abstract} In this note we compute the homology of the Lie algebra $\fgl(\infty,R)$ where $R$ is an associative unital $k$-algebra which is used in higher dimensional soliton theory \cite{Ta}.  When $k$ is a field of characteristic $0$, our result justifies an old result of Feigin and Tsygan \cite{FT}. The special case when $R=k=\C$ appeared first in soliton theory (cf. \cite{JM}). 
\end{abstract}

\maketitle

 {}\let\thefootnote\relax\footnote{{\sl Keywords:}\,  Infinite dimensional Lie algebras, Lie algebra Homology, Cyclic and Hochschild Homology, smash product, spectral sequence   \\
   \indent   {\it 2010 MSC:}\,  Primary 17B65, 16S35; Secondary 16E40.}

\begin{center}
\textit{Dedicated to the memory of our friend Jean-Louis Loday}
\end{center}
\bigskip



\setcounter{section}{-1}

\section{Introduction}
\medskip

Among several versions of the Lie algebra $\fgl$ of infinite rank, the Lie algebra $\fgl(\infty)$, that has been extensively used to describe the soliton solutions of the Kadomtsev-Petviashvili (KP in short) hierarchy (see, e.g., \cite{DJM} for detail) in the first half of 1980's, has a special feature. For example, the Lie algebra $\fgl(\infty)$ is neither ind-finite nor pro-finite. For this reason, it had been a difficult task to analyze its algebraic properties. 

In 1983, B. Feigin and B. Tsygan published a short note \cite{FT} (only 2 pages long !) where they determined the homology of the Lie algebra $\fgl(\infty, k)$ where $k$ is a field of characteristic $0$. They denoted this Lie algebra by $\fg J(k)$ that is recalled in \S \ref{sect_def-Jacobi}. Unfortunately, it seems that their paper is too dense to decompress, so that this article had not been studied  carefully in the mathematical community. At the same time, their note generated much interest, and - even 35 years later - the statements are important.

In this paper, we managed to compute the homology of the Lie algebra  $\fgl(\infty, R)$, where $R$ is an associative unital $k$-algebra and $k$ is a field of characteristic $0$. The case $R=k$ was treated by B. Feigin and B. Tsygan briefly in their note  \cite{FT} . We hope that our paper - beside generalizing the case which shows up in soliton theory - also makes the article   \cite{FT} accessible to the mathematical community.

By an argument similar to B. Feigin and B. Tsygan \cite{FT}, we have seen that the primitive part of the homology $H_\bullet(\fg J(R),k)$ is isomorphic to the cyclic homology $HC_{\bullet-1}(J(R))$, where $J(R)=\fg J(R)$ (as $k$-vector space) viewed as an associative $k$-algebra. Hence, the real problem is to express this cyclic homology in terms of the homology theory of $R$, namely, without intervention of $J$. With the aid of an analogue of the Hochschild-Serre type spectral sequence due to D. Stefan \cite{St}, 
we see that this spectral sequence degenerates at $E^2$-term. This allows us to show in Theorem \ref{thm-Hochschild-J(R)} that the homology $HH_\bullet(J(R))$ is isomorphic to $HH_{\bullet-1}(R)$. Furthermore, a detailed analysis of the above spectral sequence shows that a similar phenomena is valid, that is, the cyclic homology $HC_\bullet(J(R))$ is isomorphic to the cyclic homology $HC_{\bullet-1}(R)$.

Combining the above mentioned results, we obtain 
\[ \textrm{The primitive part of $H_\bullet(\fg J(R))$} \; = \; HC_\bullet(R)[2]\; =
   (\textrm{The primitive part of $H_\bullet(\fgl_\infty(R))$})[1],
\]
that is, $\fgl_\infty(R)$ is obtained by \textbf{delooping} $\fg J(R)$ at the homological level ! \\


Throughout this paper, a field $k$ of characteristic $0$ is fixed. 
\begin{ack}
The second author is also grateful to Jean-Louis Loday for useful discussions on this topic  he had just a few days before the tragedy. We would like to thank Alexander Voronov, Mikhail Kapranov and Max Karoubi for useful discussions. We would also like to thank the referee for her/his comments.
\end{ack}


\section{The algebra $J$}
\medskip

We recall the algebra $J$ of \cite{FT} what is called $\fgl(\infty)$ in \cite{JM}. 
We also recall some of its basic properties that will be useful for further studies in this paper. 

We adapt the convention to indicate that $J$, regarded as a Lie algebra, would be denoted as $\fg J$, and when it is viewed as an associative algebra, it would be denoted by $J$. 

Unless otherwise stated, every object is defined over a field $k$ of characteristic $0$.

\subsection{Definition} \label{sect_def-Jacobi}
Let $R$ be an associative unital $k$-algebra. 
As an $R$-module, $J(R)$ is spanned by matrices indexed over $\Z$:
\[ J(R)=\{ \,(m_{i,j})_{i,j \in \Z}\, \vert \,m_{i,j} \in R, \; \exists\,\, N \,\, \text{s.t.} \,\, m_{i,j}=0 \quad (\text{$\forall\, i,j$} \,\, \text{s.t.} \,\, \vert i-j\vert>N)\, \}.
\]
With the standard operations on matrices, $J(R)$ has a structure of associative algebra. The usual Lie bracket $[A,B]:=AB-BA$ on $J(R)$ is well-defined, and we shall denote it by $\fg J(R)$, whenever we regard it as a Lie algebra. When $R=k$ itself, we shall write $J$ and $\fg J$, for simplicity.  \\

\noindent \textbf{N.B.} \hspace{0.1in}
The reader should notice that $\fg J(R)$ is {\sl not isomorphic} to $\fg J \otimes_k R$; the latter is a proper subalgebra in $\fg J(R)$ !

\medskip

\subsection{Remark on affine Lie subalgebras}\label{sect_sub-affine}
For each $n \in \Z_{>0}$, the subalgebra of the matrices $(a_{i,j})$ with the properties $a_{i+n,j+n}=a_{i,j}$ for any $i,j$ will be denoted by $J_n(R)$.
The algebra $J_n(R)$ viewed as Lie algebra, denoted by $\fg J_n(R)$, is isomorphic to $\fgl_n(R)\otimes_k k[t,t^{-1}]$. Indeed, an isomorphism is given as follows.
For each $1\leq i,j\leq n$ and $p \in \Z$, set $e_{i,j}(p)=\sum_{r \in \Z}e_{i+rn, j+(p+r)n}$. We have
\[
[e_{i,j}(p),e_{k,l}(q)]
=\delta_{j,k}e_{i,j}(p+q)-\delta_{i,l}e_{k,j}(p+q).
\]
It is clear that $\{e_{i,j}(p)\}_{p \in \Z}$ forms a basis of $J_n(R)$ and the $R$-linear map defined by $e_{i,j}(p) \mapsto e_{i,j}\otimes t^p$ is an isomorphism of Lie algebras.

\subsection{$\fg J(R)$ versus $\fgl_n(J(R))$}\label{sect_gJ-gln(J)}
Let us  show now that the Lie algebras $\fg J(R)$ and $\fgl_n(J(R))$ are isomorphic for any $n \in \Z_{>1}$.  For any integer, say $x \in \Z$, let $x_0 \in \{ 1,2,\ldots, n\}$ be such that $x-x_0 \in n\Z$ and let $x' \in \Z$ such that $x=x_0+nx'$.

Define a $k$-linear isomorphism $\Phi_n: \fg J(R) \longrightarrow \fgl_n(J(R))$ as follows. For $\bar{i}, \bar{j} \in \Z/n\Z$, we choose $i,j \in \{1,2,\ldots, n\}$ their representatives in $\Z$. Define the map $\Phi_n: \fg J(R) \rightarrow \fgl_n(J(R))$ by $\Phi_n(M)=(M_{\bar{i},\bar{j}})_{\bar{i}, \bar{j} \in \Z/n\Z}$, where the $(r,s)$-component of the matrix $M_{\bar{i},\bar{j}}$ is given by
\[ (M_{\bar{i},\bar{j}})_{r,s}=m_{i+rn, j+sn}. 
\]
By direct computation, it can be verified that the map $\Phi_n$ is an isomorphism of Lie algebras:
\begin{lemma}\label{lemma_isom_gJ} For any $n \in \Z_{>1}$, we have
\[ \fg J(R) \cong \fgl_n(J(R)) \]
as Lie algebras.
\end{lemma}
Hence, to compute the homologies $H_\bullet(\fg J(R), k)$, it suffices to compute $H_\bullet(\fgl_\infty(J(R)), k)$, where $\fgl_\infty(\cR)$, for an associative unital $k$-algebra $\cR$, is the inductive limit of the directed family defined by \\
\centerline{$\iota_{m,n}: \fgl_m(\cR) \hookrightarrow \fgl_n(\cR); \; A=(a_{i,j})_{1\leq i,j\leq m} \mapsto \widetilde{A}=(\tilde{a}_{i,j})_{1\leq i,j\leq n}$, }

\noindent for $m<n$,  where we set
\[ \tilde{a}_{i,j}=\begin{cases} a_{i,j} \qquad &1\leq i,j\leq m, \\
                                               0 \qquad & \text{otherwise}. \end{cases}
\]
Hence, we have the next corollary due to B. Feigin and B. Tsygan:
\begin{cor}[Lemma 1 in \cite{FT}]\label{cor_isom-Hom-gl-J} 
 $H_\bullet(\fg J(R),k) \cong H_\bullet(\fgl_\infty(J(R)),k)$.
\end{cor}

In the rest of this article, we omit $k$ in the coefficients of homologies, if there seems to have no risk of confusion.
\medskip

\section{Homology of the Lie algebra $\fgl_\infty(J(R))$}\label{sect_principle}
\medskip

In this section, we relate the homology of $\fgl_\infty(J(R))$ with Hochschild homology of the algebra $J(R)$, where $R$ is an associative unital $k$-algebra. We also recall two useful relations between Hochschild homology and cyclic homology.
Unless otherwise stated, every object is defined over a field $k$ of characteristic $0$.

\subsection{Primitive part of $H_\bullet(\fgl_\infty)$}
We briefly recall a result due to B. L. Tsygan \cite{T} and independently 
by J. Loday and D. Quillen \cite{LQ}. \\

Let $\cR$ be an associative unital (maybe non-commutative) $k$-algebra. They relate the homology of $\fgl_\infty(\cR)$ with the cyclic homology of $\cR$ as follows. \\
 
A key step is show that the Chevalley-Eilenberg complex $(C_\bullet(\fgl_n(\cR),k),\partial)$ and its $\fg_n(k)$-coinvariants $(C_\bullet(\fgl_n(\cR),k)_{\fgl_n(k)},\partial)$ are quasi-isomorphic. This shows that $H_\bullet(\fgl_n(\cR),k)$ has a commutative and cocommutative DG-Hopf algebra structure. 

Since the field $k$ is of characteristic $0$, each $C_q(\fgl_n(\cR),k)$ ($q \in \Z_{>0}$) is semisimple which allows us to identify the coinvariants $C_\bullet(\fgl_n(\cR),k)_{\fgl_n(k)}$ with the invariants $C_\bullet(\fgl_n(\cR),k)^{\fgl_n(k)}$. 

Due to these facts, the primitive part of $H_\bullet(\fgl_n(\cR),k)$ can be computed
with the aid of  the first fundamental theorem on the $\fgl_n(k)$-invariants of $V^{\otimes q} \otimes (V^\ast)^{\otimes q}$ with $V=k^n$ the natural representation of $\fgl_n(k)$. The result is the following.
\begin{thm}[cf. \cite{LQ} and \cite{T}] \label{thm_LQ-T}
The primitive part of the Hopf algebra $H_\bullet(\fgl_\infty(\cR))$ is isomorphic to the cyclic homology group $HC_{\bullet-1}(\cR)$.
\end{thm}
See, e.g., Chapter 9 and 10 of \cite{Lod}, for detail.

\subsection{Connes' Periodicity Exact Sequence}\label{sect_periodicity}
Let $\cR$ be an associative unital $k$-algebra. There exists the next long exact sequence:
\begin{equation}\label{long_cyclic-Hoch}
\cdots \longrightarrow HH_n(\cR) \longrightarrow HC_n(\cR) \longrightarrow HC_{n-2}(\cR) \longrightarrow HH_{n-1}(\cR) \longrightarrow \cdots
\end{equation}
This is Theorem 2.2.1 in \cite{Lod}. The same statement is also given in \cite{T}. \\
\subsection{Connes' Bicomplex}\label{sect_bicomplex}
Here, we recall so-called Connes' bicomplex $\cB(\cR)$ for an associative unital $k$-algebra $\cR$.  {\sl Connes' bicomplex} $\cB(\cR)$ is the bicomplex with 
$\cB(\cR)_{p,q}=\cR^{\otimes q-p+1}$ if $q\geq p$ and $0$ otherwise:
\[
\UseTips
\newdir{ >}{!/-5pt/\dir{>}}
\xymatrix @=1pc @*[r]
{
     \ar[dd] & &\ar[dd] & & \ar[dd] &&  & & \\
              & &         & &          & &          & & \\
      \cR^{\otimes 3} \ar[dd]_b
      & & \ar[ll]_{\;\;\;\;\;~B_{\cR}} \cR^{\otimes 2} \ar[dd]_{b}
      & & \ar[ll]_{\;\;\;\;\;~B_{\cR}} \cR 
      & & & &   \\
              &&           &&           &&         &&   \\
      \cR^{\otimes 2} \ar[dd]_b 
      & & \ar[ll]_{\;\;\;\;\;~B_{\cR}} \cR
      & & 
      & &  & &  \\
              &&           &&           &&         && \\
      \cR& &  & & & &  & &   \\
}
\]
where $b$ is the standard boundary operator of the Hochschild complex and 
$B_{\cR}:\cR^{\otimes p+1} \rightarrow \cR^{\otimes p+2}$ is the $k$-linear map given by
\begin{align*}
B_{\cR}(r_0\otimes r_1\otimes \cdots \otimes r_p)=
&\sum_{i=0}^p (-1)^{pi}\left( 1\otimes r_i \otimes  \cdots \otimes  r_p \otimes r_0 \otimes \cdots \otimes r_{i-1} \phantom{\frac{dy}{dx}}\right. \\
&\left. \phantom{\frac{dy}{dx}(-1)^{pi}}  + r_i\otimes 1\otimes r_{i+1}\otimes \cdots \otimes  r_p \otimes  r_0\otimes  \cdots \otimes r_{i-1} \right) . \\
\end{align*}
Notice that they satisfy $b \circ B_{\cR}+B_{\cR} \circ b=0$. 
It is known (cf. Theorem 2.1.8 of \cite{Lod}) that the homologies of the total complex of this bicomplex is the cyclic homologies $HC_\bullet(\cR)$ of $\cR$.
The homology of this total complex can be also computed with the aid of the spectral sequence. By definition, its $E^1$-term is given by
\[ E_{p,q}^1=\begin{cases} \; HH_{q-p}(\cR) \qquad & q>p\geq 0, \\
                                           \; 0 \qquad & \text{otherwise},
                                           \end{cases} 
\]
and the $d^1$-map is given by the induced map $(B_{\cR})_\ast$ that can be identified with the de Rham differential when $\cR$ is commutative (cf. Proposition 2.3.3 in \cite{Lod}). 

\subsection{Summary}
Thanks to 
Lemma \ref{lemma_isom_gJ},  the homologies $H_\bullet(\fg J(R))$ are isomorphic to \\
$H_\bullet(\fgl_n(J(R)))$ for any $n \in \Z_{>1}$, in particular, we may set $n=\infty$. It turns out that $H_\bullet(\fgl_\infty(J(R)))$ admits a structure of Hopf algebra whose primitive part $\mathrm{Prim} (H_\bullet(\fgl_\infty(J(R))))$ is given by the cyclic homologies $HC_{\bullet-1}(J(R))$ by Theorem \ref{thm_LQ-T}. The sections \ref{sect_periodicity} and \ref{sect_bicomplex} indicate that if we can compute the Hochschild homologies $HH_\bullet(J(R))$, it may give us a way to determine the cyclic homologies $HC_{\bullet-1}(J(R))$. \\

In the rest of this article, we calculate the Hochschild homologies $HH_\bullet(J(R))$.
Here and after, we denote the boundary operators of Hochschild complex by $b$ that is defined as follows: for $p \in \Z_{>0}$, 
\begin{align*}
& b(r_0\otimes r_1\otimes \cdots \otimes r_p) \\
=
&\sum_{i=0}^{p-1}(-1)^{i} r_0 \otimes \cdots r_{i-1} \otimes r_ir_{i+1} \otimes \cdots \otimes r_p +(-1)^p r_pr_0\otimes r_1\otimes \cdots \otimes r_{p-1}. 
\end{align*}
 
\bigskip

\section{Hochschild-Serre type Spectral sequence}
\medskip

In this section, we explain briefly the spectral sequence obtained by D. Stefan \cite{St}. This spectral sequence is a generalization of the well-known Hochschild-Serre spectral sequence. 

Unless otherwise stated, every object is defined over a field $k$ of characteristic $0$.

\subsection{Smash Product and Examples}\label{sect_general-set-up}
Let $A$ be an associative algebra and $H$ a Hopf algebra which acts on $A$ as endomorphism.  Let $\varphi: H \rightarrow \End_k(A)$ be a morphism of $k$-algebras satisfying
\begin{equation}\label{prop-varphi}
 \varphi(h)(a_1a_2)=\sum_i (\varphi(h_{(1)})a_1)(\varphi(h_{(2)})a_2), \qquad 
 \varphi(h)(1)=\varepsilon(h)1, 
\end{equation}
where $\Delta(h)=h_{(1)} \otimes h_{(2)}$ is the Sweedler notation and $\varepsilon: H \rightarrow k$ is the counit.
Let $C$ be the {\sl smash product} $A \sharp H$ of $A$ and $H$ (cf. \cite{Sw}). That is,  we define the product structure on  $C:=A \otimes H$ by
\[ (a_1 \otimes h_1)(a_2 \otimes h_2)= a_1 (\varphi((h_1)_{(1)})a_2)\otimes (h_1)_{(2)}h_2, \]
where $\Delta(h_1)=(h_1)_{(1)} \otimes (h_1)_{(2)}$. This defines an associative structure on $C$. We remark that the algebra $C$ is an $H$-comodule algebra,  namely, it has the comodule structure given by $\Delta_C: C \rightarrow C \otimes H; a \otimes h \mapsto (a \otimes h_{(1)}) \otimes h_{(2)}$, and this map is an morphism of algebras.  \\

\medskip 

A typical example of such algebras is given by a so-called \emph{twisted group algebra}, which is defined as follows. Let $A$ be an associative unital $k$-algebra and $G$ a discrete subgroup of $k$-automorphisms of $A$. The group algebra $H=k[G]$ has a natural Hopf algebra structure with $\Delta(g)=g \otimes g$ for $g \in G$. In this case, the smash product $A\sharp H$ is the so-called twisted group algebra, denoted by $A\{G\}$ in this note. Here are two examples of  twisted group algebras: 
let $R$ be an associative unital $k$-algebra. \\

\medskip

\noindent{1)} \hspace{0.1 in} $A=A(R)=R^n=\{(a_{1},\ldots, a_{n})) \vert a_i \in R\}$, with the componentwise product structure and $H=k[\Z/n\Z]$. For $\bar{i} \in \Z/n\Z$ (we may identify it with an integer $i \in [0,n-1]$), we set $\varphi(\bar{i})(a_{1},\ldots, a_{n})=(a_{i+1}, \ldots, a_{i+n})$ where the indices are interpreted modulo $n$. \\
For $(a_1,\ldots, a_n), (b_1,\ldots, b_n) \in A(R)$ and $\bar{i},\bar{j}\in \Z/n\Z$, we have
\[
((a_1,\ldots, a_n) \otimes \bar{i})((b_1,\ldots, b_n) \otimes \bar{j}) \\
=
(a_1b_{i+1},\ldots, a_nb_{i+n}) \otimes \overline{i+j}.
\]
One sees that the algebra $C=C(R)$ is isomorphic to the algebra of $n\times n$-matrices $M_n(R)$. In fact, the isomorphism is given by
\[ (a_1,\ldots, a_n) \otimes \bar{i} \qquad \longmapsto \qquad \sum_{k=1}^n a_k e_{k,k+i}, \]
where $e_{r,s}$ is the matrix element with $1$ in the $(r,s)$-component. \\

\medskip

\noindent{2)} \hspace{0.1 in} $A=A(R)=\prod_{i \in \Z} Re_i$, where $e_i$'s are orthogonal idempotents, and $H=k[\Z]$. Here, the morphism $\varphi$ is defined by $\varphi(1)=\tau \in \End_k(A)$ where $\tau(e_i)=e_{i-1}$ ($i \in \Z$). Then, the algebra $C=C(R)$ is isomorphic to the algebra $J(R)$ via the isomorphism $e_i \otimes \tau^p \mapsto e_{i,i+p}$ for any $p \in \Z$.  Let
\[ J^{\off}(R)=\{\,(m_{i,j}) \in J(R)\, \vert \, m_{i,j} \neq 0 \quad \Rightarrow \quad i\neq j\,\} \]
be the $A(R)$-bimodule of off-diagonal part of $J(R)$. As $A(R)$-bimodule, $J(R) =A(R) \oplus J^{\off}(R)$. This description will be used in \S \ref{sect_Hochschild-computation}. Here and after, for any ring $R$, in place of saying $R$-$R$ bimodule, we shall say $R$-bimodule for simplicity, unless we consider $R$-$S$ bimodule with two different rings $R$ and $S$. 

\subsection{Stefan's Spectral Sequence}
Let us briefly recall the spectral sequence, i.e., Theorem 4 of \cite{FT}, in a down to earth manner, inspired by D. Stefan \cite{St}.
\smallskip

Having the Hochschild-Serre spectral sequence in mind, one might guess that the Hochschild homology of $C=A \sharp H$ with coefficients in a $C$-bimodule 
can be described in terms of the Hochschild homologies of $A$ and $H$.
This is what happens that was proved by D. Stefan \cite{St} in a slightly general form.
Here, we recall the result in a weaker form that is sufficient for our purpose:
\begin{thm}[Theorem 4.5 in \cite{St}]\label{thm_HS-spectral-seq}Let $M$ be a $C$-bimodule. Then, there exists a convergent spectral sequence such that
\[ E_{p,q}^2 =H_p(H,H_q(A,M)) \qquad \Longrightarrow \qquad H_{p+q}(C,M).
\]
\end{thm}
In particular, the case when $M=C$ was stated as Theorem 4 in \cite{FT}, but their description contains some gaps and the proof is not given. \\
\smallskip

For the sake of reader's convenience, let us describe the $H$-bimodule structure on $H_\bullet(A,M)$ in the case when $H$ is cocommutative. \\
The left $H$-module structures are described as follows: for $h \in H$, 
\begin{enumerate}
\item on $A$: \quad $h.a:=\varphi(h)(a)$ \qquad $a \in A$, 
\item on $H$: \quad $h.m:=h_{(1)}mS(h_{(2)})$ \qquad $m \in M$.
\end{enumerate}
The right module structure on $A$ and $M$ are given by the counit $\varepsilon$ of $H$.  The $H$-bimodule structure on the Hochschild complex $C_\bullet(A,M)$ is defined as an appropriate tensor product of these modules. It can be checked that the action commutes with the boundary operator $b$ on $C_\bullet(A,M)$. Thus, this induces an $H$-bimodule structure on $H_\bullet(A,M)$. Moreover, one can also verify that $H_0(H,H_0(A,M))\cong H_0(C,M)$ as in the proof of Proposition 4.2 in \cite{St}. Thus, the $E^2$-term in Theorem \ref{thm_HS-spectral-seq} should be calculated with respect to this induced $H$-bimodule structure on $H_\bullet(A,M)$. 

Now, let $R$ be an associative unital $k$-algebra. In our case, we have $A=A(R)=\prod_{i \in \Z} Re_i$ with orthogonal idempotents $\{e_i\}_{i \in \Z}$, i.e., $e_ie_j=\delta_{i,j}e_i$ and $H=k[\Z]=k[\tau^{\pm 1}]$ where $\tau$, as an element of $\mathrm{Aut}_k(A)$, is realized as $\tau(e_i)=e_{i-1}$. Hence, we can use the spectral sequence: $E_{p,q}^2=H_p(k[\Z], H_q(A(R),J(R))) \; \Rightarrow\; HH_{p+q}(J(R))$.

\bigskip

\section{Computations on $H_\bullet(\prod_{i\in \Z} R e_i, J(R))$}\label{sect_Hochschild-computation}
\medskip

In this section we compute the homologies $H_\bullet(A(R), J(R))$ with $A(R)=\prod_{i\in \Z} R e_i$, where $R$ is an associative unital $k$-algebra. See \S \ref{sect_general-set-up} for the other notations.


\subsection{$H_0(A(R),J(R))$}
By definition, $H_0(A(R),J(R))=J(R)/[A(R),J(R)]$. For $D=\diag(d_i) \in A(R)$ and $M=(m_{i,j})\in J(R)$, one has $[D,M]_{i,j}=d_im_{i,j}-m_{i,j}d_j$, i.e., 
\[ \mathrm{Im} \,b\vert_{J(R) \otimes A(R)}=\{M=(m_{i,j}) \in J(R) \vert m_{i,i}\in [R,R] ~ \forall\, i\}, 
\]
which implies 
\begin{equation}\label{eq-Hochschild-AJ}
H_0(A(R),J(R)) \cong A(R^{ab}),
\end{equation} 
where we set $R^{ab}:=R/[R,R]=HH_0(R)$. 
\subsection{$HH_p(A(R))$ ($p>0$)}\label{sect_A-higher}
We recall Theorem 9.1.8 of \cite{W} which states that for an $R_i$-bimodule $M_i$ ($i=1,2$), where $R_i$ are associative $k$-algebras, one has
\begin{equation}\label{eq-prod-Hochschild}
 H_\bullet(R_1\times R_2, M_1\times M_2) \cong H_\bullet(R_1,M_1)\oplus H_\bullet(R_2,M_2). 
\end{equation}

For $N\in \Z_{\geq 0}$, set $A_N(R) =\prod_{\vert i\vert \leq N} Re_i$. Then for any $M>N$, the canonical projection $A_M(R) \twoheadrightarrow A_N(R)$ is surjective and it induces a surjection between the Hochschild complices $C_\bullet(A(R),A_M(R))$ and $C_\bullet(A(R),A_N(R))$ that is also surjective. Hence, the {\sl Mittag-Leffler condition} 
(cf. see, e.g., \cite{KS} or \cite{W}) for $\{C_\bullet(A(R),A_N(R))\}_{N \in \Z_{\geq 0}}$
is satisfied. 

It follows from \eqref{eq-prod-Hochschild} that $H_\bullet(A(R),A_N(R)) \cong HH_p(A_N(R)) \cong \prod_{\vert i\vert \leq N} HH_p(R)e_i$. It follows that the Mittag-Leffler condition for $\{H_\bullet(A(R),A_N(R))\}_{N \in \Z_{\geq 0}}$ is also satisfied. 
Thus, by Proposition 1.12.4 of \cite{KS} or Theorem 3.5.8 of \cite{W}, it follows that 
$HH_\bullet(A(R)) \cong \underset{\longleftarrow_N}{\lim} HH_\bullet(A_N(R))$. 
Therefore, we obtain
\begin{equation}\label{eq-Hochschild-A}
 HH_\bullet(A(R)) \cong \prod_{i \in \Z} HH_\bullet(R)e_i.
\end{equation}

\begin{rem}\label{rem_R=k-I} Let us suppose that $R=k$. 
So, we compute $HH_\bullet(ke)$ for an idempotent $e$, i.e., $e^2=e$. By definition, for $r>0$, 
\[ b(e\otimes e^{\otimes r})=\begin{cases} 0 \qquad &r \equiv 1 (2), \\
                                                                  e\otimes e^{\otimes (r-1)} \qquad & r \equiv 0 (2), \end{cases} \]
which implies $HH_0(ke) =ke$ and $HH_p(ke)=0$ for $p>0$.
Hence, \eqref{eq-Hochschild-A} implies
\[
 HH_p(A(k)) \cong \begin{cases} A(k) \qquad &p=0, \\ 0 \qquad & p>0. \end{cases}
\]
\end{rem}
\smallskip

\subsection{$H_p(A(R),J^{\off}(R))$} We first observe that $R e_{i,j} \in J^{\off}(R)$ for any $i\neq j \in \Z$ is an $A(R)$-bimodule. \\

Let $I_1,I_2$ be subsets of $\Z$ satisfying i) $i \in I_1$ and $j \in I_2$, ii) $I_1 \cup I_2=\Z$ and iii) $I_1 \cap I_2=\emptyset$.  For any subset $I \subset \Z$, we set $A(R)_I=\prod_{i \in I} Re_i$ and $e_I=\sum_{i \in I} e_i$. It is clear that $e_I$ is the unit of $A(R)_I$. Moreover, the subalgebra $S=ke_{I_1}\oplus ke_{I_2}$ of $A(R)$ is separable over $k$, since we can take $e_{I_1} \otimes e_{I_1}+e_{I_2} \otimes e_{I_2}$ as an idempotent. Thus, by Theorem \ref{thm_non-comm-Hochschild}, we have
\[ H_p(A(R), Re_{i,j}) \cong H_p^S(A(R),Re_{i,j}). \]
But since $S$ is central, it follows that $A(R)^{\otimes_S p}=A(R)_{I_1}^{\otimes_S p} \oplus A(R)_{I_2}^{\otimes_S p}$. This implies that $C_p^S(A(R), Re_{i,j})=0$, thus we obtain $H_p(A(R),Re_{i,j})=0$.

Now, let $\{I_i\}_{i \in \Z_{>0}}$ be an increasing sequence (, i.e, $I_i \subset I_{i+1}$ for any $i$) of finite subsets of $\Z$ such that $\bigcup_i I_i =\Z$. Set $J(R)_i=\bigoplus_{p \in \Z \setminus \{0\}} \left(\prod_{j \in I_i}Re_j\right)\otimes k \tau^p$. By definition, $J(R)_i$ is an $A(R)$-bimodule and $\underset{\longleftarrow}{\lim}_{i} J(R)_i=J^{\off}(R)$. Moreover, both $\{J(R)_i\}_{i \in \Z_{>0}}$ and $\{H_\bullet(A(R),J(R)_i\}_{i \in \Z_{>0}}$ satisfy the Mittag-Leffler condition. Thus, by Proposition 1.12.4 of \cite{KS} or Theorem 3.5.8 of \cite{W}, it follows that 
\begin{equation}\label{eq-Hochschild-A-J(1)}
H_p(A(R), J^{\off}(R))=0 \qquad \forall \; p>0. 
\end{equation}

Combining \eqref{eq-Hochschild-AJ}, \eqref{eq-Hochschild-A} and \eqref{eq-Hochschild-A-J(1)}, we obtain the following result.
\begin{prop}\label{prop_Homology-A}
We have $H_\bullet(A(R),J(R))\cong \prod_{i \in \Z} HH_\bullet(R)e_i.$  
\end{prop}

In the next section, we shall thus compute the homologies $H_\bullet(k[\Z],\prod_{i \in \Z} HH_\bullet(R)e_i)$. 

\bigskip

\section{Computations of $H_\bullet(k[\Z],\prod_{i \in \Z} HH_\bullet(R)e_i)$}\label{sect_H(H,A)}
\medskip

Let us recall its set up. Our $k[\Z]$ is generated by the matrix $\tau=\sum_{i \in \Z} e_{i,i+1} \in J$ which acts on $M_\bullet(R):=\prod_{i \in \Z}HH_\bullet(R) e_i$ from the left by conjugation, i.e., $\tau(\sum_{i}m_i e_i)=\sum_i m_i e_{i-1}$. The right module structure should be given by the counit, i.e., $\tau \mapsto \, \text{the multiplication by}\, 1$. This $k[\Z]$-bimodule structure is the same as is given in \cite{St}. 

Since the global dimension of a principal ideal domain, which is not a field,  is $1$, it suffices to compute the homologies $\{H_p(k[\Z],M_\bullet(R))\}_{p \in \Z_{\geq 0}}$ for $p=0$ and $1$. \\

\subsection{$H_0(k[\Z],M_\bullet(R))$}

The boundary map $b:M_\bullet(R)\otimes k[\Z] \rightarrow M_\bullet(R)$ is given by $e_i \otimes \tau^p \mapsto e_i-e_{i-p}$. Hence, for any $\sum_{i \in \Z}m_i e_i \in M_\bullet(R)$, setting 
\[ \tilde{m}_i =\begin{cases} \sum_{0<r\leq i}m_r \qquad & i>0, \\
                                 \qquad 0 \qquad &i=0, \\
                                 -\sum_{i<r\leq 0}m_r \qquad &i<0, \end{cases}
\]
it can be checked that 
\[   b(\sum_i \tilde{m}_i e_i\otimes \tau^{-1})=\sum_i \tilde{m}_i (e_i-e_{i+1})=\sum_i(\tilde{m}_i-\tilde{m}_{i-1})e_i=\sum_i m_ie_i.
\]
This implies that $\mathrm{Im}\,b\vert_{M_\bullet(R)\otimes k[\Z]}=M_\bullet(R)$. Thus we obtain the following.

\begin{lemma} $H_0(k[\Z],M_\bullet(R))=0$.
\end{lemma}
\subsection{$H_1 (k[\Z],M_\bullet(R))$}

For $m=\sum_{i \in \Z}m_ie_i \in M_\bullet(R)$ and $p \in \Z$,  set 
\[ m[p]=\tau^p(m)=\sum_i m_ie_{i-p}=\sum_i m_{i+p}e_i. 
\]
The next lemma is technical, but it  simplifies the rest of the computation.
\begin{lemma} Any element of $H_1(k[\Z],M_\bullet(R))$ is represented in the form $m \otimes \tau$ for some $m \in M_\bullet(R)$. 
\end{lemma}
\begin{proof} First of all, remark that $m \otimes 1 \equiv 0$ in $H_1(k[\Z],M_\bullet(R))$ for any $m \in A$, since $b(m\otimes 1 \otimes 1)=m\otimes 1$.
Now, for $p,q \in \Z$, one has
\[ b(m\otimes \tau^p \otimes \tau^q)=m\otimes \tau^q-m\otimes \tau^{p+q}+m[q]\otimes \tau^p, \]
which implies, setting $q=1$, that
\[ m \otimes \tau^p\equiv (m+m[1]+\cdots +m[p-1])\otimes \tau  \qquad \text{in} \qquad H_1(k[\Z],M_{\bullet}(R))\]
for $p>0$ (by induction) and, setting $q=-p$, that
\[ m\otimes \tau^{-p}\equiv -m[-p]\otimes \tau^p \qquad \text{in} \qquad H_1(k[\Z],M_{\bullet}(R)).
\]
for any $p>0$. 
\end{proof}

\begin{rem} \label{rem_H1}
We have seen that in $H_1(k[\Z], M_\bullet(R))$ one has
\[ m \otimes \tau^p\equiv \begin{cases} 
(\sum_{k=0}^{p-1} m[k]) \otimes \tau \qquad &p>0, \\
\qquad 0 \qquad  &p=0, \\
-(\sum_{k=p}^{-1} m[k])\otimes \tau \qquad &p<0.
\end{cases}
\]
\end{rem}

Hence, we may restrict ourselves to consider the elements of type $m\otimes \tau$ ($m \in M_\bullet(R)$).
\medskip
By definition, we have $b(m \otimes \tau)=m-m[1]$, which implies that $m \otimes \tau \in \Ker\, b$ iff $m=m[1]$, i.e., $m \in HH_\bullet(R)(\sum_{i \in \Z}e_i)$. Hence, we see that there is a surjective map $HH_\bullet(R)(\sum_{i \in \Z} e_i)\otimes \tau \twoheadrightarrow H_1(k[\Z],M_\bullet(R))$. This map is also injective, since any element of $M_\bullet(R) \otimes k[\Z] \otimes k[\Z]$ is a linear combination of the elements of the form $m\otimes \tau^p \otimes \tau^q$, and
its image by $b$ can be computed as
\[
b(m\otimes \tau^p \otimes \tau^q)=
m\otimes \tau^q-m\otimes \tau^{p+q}+m[q]\otimes \tau^p \equiv 0 \otimes \tau \in H_1(k[\Z],M_\bullet(R))
\]
by Remark \ref{rem_H1}. Thus, we obtain
\begin{prop}\label{prop_Homology-k[Z]-A} $H_1(k[\Z], M_\bullet(R)) \cong HH_\bullet(R)$ and $H_p(k[\Z],M_\bullet(R))=0$ for $p\neq 1$.
\end{prop}

\bigskip


\section{Main results}\label{sect_main}
\medskip

We describe now the homology groups of $\fg J(R)$ over an arbitrary associative unital $k$-algebra $R$ that is not necessarily commutative. 

\subsection{Hochschild Homology of $J(R)$}
Theorem \ref{thm_HS-spectral-seq} together with Propositions \ref{prop_Homology-A} and \ref{prop_Homology-k[Z]-A} show that the Hochschild homologies of $J(R)$ are determined since the spectral sequence $E_{p,q}^2=H_p(k[\Z], H_q(A(R),J(R))) \; \Rightarrow\; HH_{p+q}(J(R))$ collapses at $E^2$ and it gives
\begin{thm} \label{thm-Hochschild-J(R)}
Suppose that $R$ is an associative unital $k$-algebra. Then, 
$HH_\bullet(J(R)) \cong HH_{\bullet-1}(R)$ as graded algebras.
\end{thm}

For later purpose, let us provide an explicit isomorphism between $HH_p(R)$ and $HH_{p+1}(J(R))$.  Such a morphism is constructed via the composition $HH_p(R) \cong H_1(k[\Z],\prod_{i \in \Z}HH_p(R)e_i) \cong HH_{p+1}(R)$, where the second map is given by ''shuffle product'' (cf. \cite{EM} or \cite{M}). To be explicit, this isomorphism is induced from the morphism of abelian group
$\widetilde{\Phi}_p:R^{\otimes p+1} \rightarrow J(R)^{\otimes p+2}$ defined by
\begin{equation}\label{eq-def-map-Phi-I}
\widetilde{\Phi}_p(r_0\otimes r_1\otimes \cdots \otimes r_p)=
r_0 I\otimes \left( \sum_{k=0}^p (-1)^k r_1 I\otimes \cdots \otimes r_k I\otimes I \tau \otimes r_{k+1}I \otimes \cdots \otimes r_p I\right),
\end{equation}
for $p \in \Z_{\geq 0}$, where we set $I=\sum_{i \in \Z} e_i$ and $\tau \in k[\Z]$ is defined in the head of \S \ref{sect_H(H,A)}. Indeed, by direct computation, one obtains the next lemma :
\begin{lemma}\label{lemma_rel-b-Phi-I} For any $p \in \Z_{\geq 0}$, one has
$b \circ \widetilde{\Phi}_{p+1}+\widetilde{\Phi}_{p}\circ b=0$. 
\end{lemma}
Hence, by this lemma, the morphism $\widetilde{\Phi}_p$ induces a morphism $\Phi_p: HH_p(R) \rightarrow HH_{p+1}(J(R))$ for any $p \in \Z_{\geq 0}$. 

The next proposition seems to be well-known  (cf. \cite{HS1} and \cite{HS2}):
\begin{prop}\label{prop-isom-Phi}
For any $p \in \Z_{\geq 0}$, the morphism $\Phi_p$ is an isomorphism.
\end{prop}
\subsection{Homology of the Lie algebra $\fg J(k)$}
As we have seen in Remark \ref{rem_R=k-I}, it follows that $HH_0(k)=k$ and $HH_p(k)=0$ for $p>0$. Hence, by Theorem \ref{thm-Hochschild-J(R)},
we obtain 

\begin{cor}[Theorem 3 in \cite{FT}]\label{cor_FT-III}
$HH_1(J) \cong k$ and $HH_p(J)=0$ for any $p\neq 1$. 
\end{cor}

In this case, Corollary \ref{cor_FT-III} together with the Connes' periodicity exact sequence \eqref{long_cyclic-Hoch} in \S \ref{sect_periodicity} implies 
the next periodicity:
\begin{cor} $HC_p(J) \cong k$ for odd $p$ and $HC_p(J)=0$ for even $p$. 
\end{cor}
Hence, by Corollary \ref{sect_gJ-gln(J)} and the Loday-Quillen-Tsygan theorem (cf. Theorem \ref{thm_LQ-T}), the next theorem 
follows from the theorem of Milnor-Moore \cite{MM}:
\begin{thm}\label{thm-FT_I-dual} There exist primitive elements $x_i \in H_{2i}(\fg J,k)$ for any $i \in \Z_{>0}$ such that the homology $H_\bullet(\fg J,k)$ is isomorphic to the Hopf algebra $S( \bigoplus_{i \in \Z_{>0}}k x_i)$.
\end{thm}
The dual statement to this theorem is due to B. Feigin and B. Tsygan:
\begin{thm}[Theorem 1 a) in \cite{FT}]\label{thm-FT-I}
There exist primitive elements $c_i \in H^{2i}(\fg J,k)$ for any $i \in \Z_{>0}$ such that the homology $H^\bullet(\fg J,k)$ is isomorphic to the Hopf algebra $S( \bigoplus_{i \in \Z_{>0}}k c_i)$.
\end{thm}

\subsection{Cyclic Homology of $J(R)$} 
Here, we determine the cyclic homology $HC_\bullet(J(R))$ of $J(R)$, where $R$ is an associative unital $k$-algebra.
For this purpose, we compute the total complex of the Connes bicomplex $\cB(J(R))$ recalled in \S \ref{sect_bicomplex}, with the aid of the spectral sequence. By definition, its $E^1$-term is given by
\[ E_{p,q}^1=\begin{cases} \; HH_{q-p}(J(R)) \qquad & q>p\geq 0, \\
                                           \; 0 \qquad & \text{otherwise}.
                                           \end{cases} 
\]
By Theorem \ref{thm-Hochschild-J(R)}, one may expect that the induced maps $(B_{J(R)})_\ast$ and $(B_{R})_\ast$ are related via the isomorphisms $\Phi_p$ :
 
\begin{lemma}\label{lemma_rel-b-Phi-II} For any $p \in \Z_{\geq 0}$, one has 
$\Phi_{p+1} \circ (B_R)_\ast +(B_{J(R)})_\ast \circ \Phi_p=0$. 
\end{lemma}
Indeed, one can show that, for a cycle $\omega \in R^{\otimes p+1}$, i.e., $b(\omega)=0$, it follows that
\[  (\widetilde{\Phi}_{p+1} \circ B_R + B_{J(R)} \circ \widetilde{\Phi}_p)(\omega)
=-b(1 \otimes \tau \otimes 1 \otimes (N(\omega))), \]
where $N=\sum_{k=0}^p t^k$ with $t.(r_0\otimes r_1\otimes \cdots r_p)=(-1)^p r_p \otimes r_0\otimes \cdots \otimes r_{p-1}$ is a linear map. 
\medskip

By this lemma, together with Proposition 2.3.3 in \cite{Lod}, one may regard the $E^1$-term of the above spectral sequence as
\[ E_{p,q}^1=\begin{cases} \; HH_{q-p-1}(R) \qquad & q>p\geq 0, \\
                                           \; 0 \qquad & \text{otherwise},
                                           \end{cases} 
\]
whose $d^1$-map is given by the induced map $-(B_R)_\ast$ (or $0$). This means that the total complex of $\cB(J(R))$ is quasi-isomorphic to the total complex of $\cB(R)$ shifted by $[1]$, that is, we obtain the next theorem:
\begin{thm}\label{thm-main-I}
Let $R$ be an associative unital $k$-algebra. Then, $HC_{\bullet}(J(R)) \cong HC_{\bullet}(R)[1]$ as graded $k$-vector spaces.
\end{thm}

Therefore, the Loday-Quillen-Tsygan theorem (cf. Theorem \ref{thm_LQ-T}) and the  Quillen version \cite{Q} of the Milnor-Moore theorem yield the next result:
\begin{thm}\label{thm-main-general}
$H_\bullet(\fg J(R)) \cong S(HC_\bullet(R)[2])$ as graded Hopf-algebras.
\end{thm}

\begin{rem}\label{rem_delooping} 
By the Loday-Quillen-Tsygan theorem (cf. Theorem \ref{thm_LQ-T}) and Theorem \ref{thm-main-I}, we observe that 
\[ \mathrm{Prim}\,( H_\bullet(\fg J(R))) \cong \mathrm{Prim}\,(H_{\bullet}(\fgl_\infty(R)))[1] \cong HC_{\bullet}(R)[2]. \]
Therefore, the difference between $\fg J(R)$ and $\fgl_\infty(R)$ is by no means trivial.
\end{rem}

The orthogonal and symplectic versions of our results are obtained in \cite{FI} as an application of the results this article.

\subsection{Universal Central Extension}\label{sect_UCE}
The Lie algebra $\fg J(R)$ begin perfect, it allows the universal central extension (cf. Chapter 7 in \cite{W}). In this subsection, we describe it explicitly.

Theorem \ref{thm-main-general} implies
\begin{cor} $H_2(\fg J(R))=HC_0(R)=R^{ab}=R/[R,R]$.
\end{cor}
Remark that this result also follows from Theorem \ref{thm-Hochschild-J(R)} and Connes' Periodicity Exact Sequence \eqref{long_cyclic-Hoch}.

Set $I_+=\sum_{i\geq 0} e_{i,i}$ and $I_-=\sum_{i<0} e_{i,i}$. It is clear that $I_{\pm} \in J(R)$ and $I_\sigma I_\tau=\delta_{\sigma, \tau} I_\tau$ for $\sigma, \tau \in \{\pm\}$. Let $\Phi:J(R) \rightarrow J(R)$ be the $k$-linear map defined by $\Phi(X)=I_+XI_+$.  A matrix $M=(m_{i,j}) \in J(R)$ is said to be of finite support if the set $\{(i,j)\, \vert m_{i,j} \neq 0\}$ is finite. We denote the Lie subalgebra of $\fg J(R)$ consisting of the matrices of finite support by $\fg F(R)$.  Define the trace map $\Tr: \fg F(R) \rightarrow R^{ab}$ as the composition of the usual trace map $\mathrm{tr}: \fg F(R) \rightarrow R; M=(m_{i,j})\, \mapsto \sum_i m_{i,i}$ and the abelianization $\pi^{ab}: R \twoheadrightarrow R^{ab}$. 

Now, let $\Psi: J(R) \times J(R) \rightarrow R^{ab}$ be the $k$-bilinear map defined by
\begin{align*}
\Psi(X,Y)=
&\Tr([\Phi(X), \Phi(Y)]-\Phi([X,Y])) \\
=
&\Tr((I_+YI_-)(I_-XI_+)-(I_+XI_-)(I_-YI_+)).
\end{align*}
It can be checked that
\begin{lemma} The $k$-bilinear map $\Psi$ is a $2$-cocycle, i.e., 
it satisfies
\begin{enumerate}
\item $\Psi(Y,X)=-\Psi(X,Y)$, 
\item $\Psi([X,Y],Z)+\Psi([Y,Z],X)+\Psi([Z,X],Y)=0$,
\end{enumerate}
for any $X,Y,Z \in J(R)$. 
\end{lemma}
This $2$-cocycle is called {\sl Japanese cocycle}. Let $\widetilde{\fg J}(R)$ be the universal central extension of the Lie algebra $\fg J(R)$. 
\begin{thm} The Lie algebra $\widetilde{\fg J}(R)$ is a $k$-vector space
\[ \widetilde{\fg J}(R)=\fg J(R) \oplus R^{ab} \]
equipped with the Lie bracket $[ \cdot, \cdot ]'$ defined by 
\[
[\widetilde{\fg J}(R), R^{ab}]'=0, \qquad
[X,Y]'=[X,Y]+\Psi(X,Y) \quad X,Y \in \fg J(R).
\]
\end{thm}
\subsection{The case of smooth commutative algebras}\footnote{The notion of smoothness for commutative algebras is more general than the quasi-freeness introduced by Cuntz and Quillen in \cite{CQ}. Indeed, a smooth commutative algebra may have any positive finite cohomological dimension whereas the cohomological dimension of a quasi-free commutative algebra can be at most $1$.}

Here, we suppose that $R$ is a {\sl smooth commutative algebra} over $k$, i.e., $R$ is a commutative unital $k$-algebra such that, for any maximal ideal $\fm \subset R$,  the kernel of the localized map $\mu_{\fm}: (R \otimes_kR)_{\mu^{-1}(\fm)} \rightarrow R_{\fm}$, where $\mu:R \otimes_k R \rightarrow R$ is the multiplication map, is generated by a regular sequence in $(R \otimes_kR)_{\mu^{-1}(\fm)}$ (cf. \S 3.4 in \cite{Lod}).

The Hochschild-Kostant-Rosenberg theorem \cite{HKR} asserts that $HH_{\bullet}(R)$ is isomorphic to $\Omega_{R\vert k}^\bullet$ as graded algebra. Hence, Theorem \ref{thm-Hochschild-J(R)} implies
\begin{cor}\label{cor-HKR-version}
$HH_\bullet(J(R)) \cong \Omega_{R\vert k}^{\bullet-1}$ as graded algebras. 
\end{cor}

By Theorem \ref{thm-main-I} and Theorem 3.4.12 in \cite{Lod}, we have
\begin{thm}\label{thm-main-HC} Suppose that $R$ is smooth over $k$. Then, there is a canonical isomorphism
\[ HC_p(J(R)) \cong \Omega_{R\vert k}^{p-1}/d\Omega_{R\vert k}^{p-2} \oplus \bigoplus_{\substack{0\leq r<p-2 \\ r \equiv p-1 [2]}} H_{DR}^{r}(R),
\]
where $H_{DR}^r(R)$ signifies the de Rham cohomology of $R$. 
\end{thm}



\subsection{Homology of some Lie subalgebras of $\fg J(R)$}
Let $R$ be an associative unital $k$-algebra. Set
\begin{align*}
J_+ (R)=
&\{ M=(m_{i,j}) \in J(R) \vert \; m_{i,j}=0 \quad (\text{$\forall\, i,j$} \,\, \text{s.t.} \,\,  i \leq 0\, \text{or}\, j\leq 0) \, \}, \\
J^\geq(R)=
&\{ M=(m_{i,j}) \in J(R) \vert \; \exists\,\, N \,\, \text{s.t.} \,\, m_{i,j}=0 \quad (\text{$\forall\, i,j$} \,\, \text{s.t.} \,\, i<N \, \text{or}\,  j<N)\}.
\end{align*}
Via the isomorphism $\varphi$ of $J(R)$ defined by $\varphi(e_{i,j})=e_{-i,-j}$, where $e_{i,j}$ is the matrix unit, we also set $J_-(R)=\varphi(J_+(R))$ and $J^\leq (R)=\varphi(J^\geq(R))$. As before, the algebras $J_\pm(R), J^{\geq}(R)$ and $J^{\leq}(R)$, viewed as Lie algebras, will be denoted by $\fg J_\pm(R), \fg J^{\geq}(R)$ and $\fg J^{\leq}(R)$, respectively.
We have
\begin{lemma} \label{lemma_isom-homologies}
$H_\bullet(\fg J_+(R), k) \cong H_\bullet(\fg J_-(R),k) \cong H_\bullet(\fg J^{\leq}(R),k) \cong H_\bullet(\fg^{\geq}(R),k)$.
\end{lemma}
\begin{proof} Since the first and the third isomorphisms are induced from the isomorphism $\varphi$, let us show the second isomorphism. 
For $n \in \Z$, set
\[ J^{\leq n}(R)=\{ M=(m_{i,j}) \in J(R) \vert \; m_{i,j}=0 \quad (\text{$\forall\, i,j$} \,\, \text{s.t.} \,\,  i> n \, \text{or}\, j> n) \, \}.
\]
By definition, it follows that $J^{\leq }(R) \cong \underset{\longrightarrow}{\lim} J^{\leq n}(R)$ via the natural embedding $J^{\leq m}(R) \hookrightarrow J^{\leq n}(R)$ ($m<n$). As $J^{\leq n}(R)$ is isomorphic to $J_-(R)$, this implies that
\[ H_\bullet(\fg J^{\leq}(R),k) \cong \underset{\longrightarrow}{\lim} \, H_\bullet(\fg J^{\leq n}(R),k) \cong H_\bullet(\fg J_-(R),k), \]
where $\fg J^{\leq n}(R)$ is the $k$-vector space $J^{\leq n}(R)$ viewed as Lie algebra.
\end{proof}
An analogous statement of Corollary \ref{cor_isom-Hom-gl-J} holds, namely, one has
\[ H_\bullet(\fg J^{\ast}(R)) \cong H_\bullet(\fgl_\infty(J^{\ast}(R))) \qquad \ast \in \{\geq, \leq\}. 
\]
Hence, the Loday-Quillen-Tsygan theorem (cf. Theorem \ref{thm_LQ-T}) implies that
it suffices to calculate $HC_{\bullet-1}(J^{\ast}(R))$ ($\ast \in \{\geq, \leq\}$). For this purpose, we first determine the Hochschild homologies $HH_\bullet(J^{\ast}(R))$  ($\ast \in \{\geq, \leq\}$). 
Recall that the subalgebras $J^{\geq }(R)$ and $J^{\leq}(R)$ of $J(R)$ can be viewed as twisted group algebras as follows: set $A^{\pm}(R)=\prod_{\pm i\geq 0}R e_i \oplus \bigoplus_{\pm i<0}R e_i$. Then, it can be checked that $J^{\geq}(R) \cong A^{+}(R)\sharp H$ and $J^{\leq}(R) \cong A^{-}(R)\sharp H$.
By \S \ref{sect_general-set-up} and Theorem \ref{thm_HS-spectral-seq}, it suffices to compute the spectral sequence \\
\centerline{$E_{p,q}^2=H_p(k[\Z], H_q(A^+(R),J^{\geq}(R))) \quad \Longrightarrow \quad HH_{p+q}(J^{\geq}(R))$.}

\smallskip

\noindent  It can be shown, as in \S \ref{sect_Hochschild-computation}, that
\[ H_\bullet(A^+(R),J^{\geq}(R)) \cong \prod_{i \geq 0}HH_\bullet(R) e_i \oplus \bigoplus_{i<0} HH_\bullet(R)e_i. \]
Denote the right hand side of this formula by $M_\bullet(R)$ as in \S \ref{sect_H(H,A)}. As $k[\Z]$ is a PID, $E_{p,q}^2=0$ unless $p=0,1$. By definition, 
one has $b(m \otimes \tau^p)=m-m[p]$ which implies that $\Ker\vert_{M_\bullet(R)\otimes k[\Z]}=M_\bullet(R) \otimes k \tau^0$ and 
$\mathrm{Im}\vert_{M_\bullet(R)\otimes k[\Z]}= M_\bullet(R)$.  Since $b(m \otimes 1 \otimes 1)=m\otimes 1$ again by definition,  we obtain 
\[ H_p(k[\Z], M_\bullet(R))=0 \qquad p \in \{0,1\}. 
\]
Thus, we obtain $HH_\bullet(J^{\geq}(R))=0$. In particular, this implies $HC_\bullet(J^{\geq}(R))=0$. 
\begin{thm} For $p \in \Z_{\geq 0}$, 
\[ H_p(\fg J_+(R),k) \cong H_p(\fg J_-(R),k)\cong H_p(\fg J^{\leq}(R),k)
\cong H_p(\fg J^{\geq}(R),k) \cong \begin{cases} 
\; k \qquad & p=0, \\ \; 0 \qquad & p>0. \end{cases}
\]
\end{thm}

\medskip 

\appendix 

\section{Hochschild Homology over Non-commutative Ground Rings}\label{sect_Hochschild-non-comm}
Here, we generalize \S 1.2 of Loday's book \cite{Lod}. 

Let $R$ be an associative unital (bot not necessarily commutative) ring and
$S$ be its subring. For an $R$-bimodule $M$, the group of $p$-chains $C_p^S(A,M)$ is $M\otimes_S A^{\otimes_S p}\otimes_S$ which signifies the quotient of $M\otimes_S A^{\otimes_S p}$ by the relation
$m\otimes a_1\otimes \cdots \otimes a_p.s=s.m\otimes a_1\otimes \cdots \otimes a_p$ for any $s \in S$, $m \in M$ and $a_i \in A$. 
For instance, $R\otimes_S=R/[S,R]$. It is clear that the boundary map $b$ of the Hochschild complex induces a boundary map on $C_\bullet^S(R,M)$ which we denote it again by $b$. Its Homology will be denotes as $H_\bullet^S(R,M)$. \\

Suppose that a unital $k$-algebra $S$ is {\sl separable} over $k$, i.e., the $S$-bimodule map $\mu: S \otimes S^{\op} \rightarrow S$ splits. This is equivalent to the existence of an idempotent $e=\sum_i u_i \otimes v_i \in S \otimes S^{\op}$ such that $\sum_i u_iv_i=1$ and $(s\otimes 1)e=(1\otimes s)e$ for any $s\in S$.

Now, we can show the next theorem:
\begin{thm}\label{thm_non-comm-Hochschild}
Let $S$ be a separable algebra over $k$ and $R$ an associative unital $S$-algebra.
Then, for any $R$-bimodule $M$, there is a canonical isomorphism:
\[ H_\bullet(R,M) \cong H_\bullet^S(R,M). \]
\end{thm}
\begin{proof} Let $\phi: C_\bullet(R,M) \twoheadrightarrow C_\bullet^S(R,M)$ be a canonical projection. Write $e=\sum_i u_i \otimes v_i \in S\otimes_k S^{\op}$. We define the $k$-linear map $\psi:C_\bullet^S(R,M) \rightarrow C_\bullet(R,M)$ by 
\[ \psi(m \otimes r_1\otimes \cdots \otimes r_p)=\sum_{i_0,i_1,\ldots, i_p}
v_{i_p}mu_{i_0}\otimes v_{i_0}r_1u_{i_1}\otimes v_{i_1}r_2u_{i_2}\otimes \cdots \otimes v_{i_{p-1}}r_pu_{i_p}. \]
By definition, it follows that $\phi \circ \psi=\id_{C_\bullet^S(R,M)}$. Hence, let us compute $\psi \circ \phi$. Unfortunately, this is not the identity operator, but one can construct an explicit homotopy to the identity operator as follows. 
For $0\leq i\leq p$, we set
\begin{align*}
&h_i(m\otimes r_1\otimes \cdots \otimes r_p)\\
=
&\sum_{j_0,j_1,\ldots, j_i}
mu_{j_0}\otimes v_{j_0}r_1u_{j_1}\otimes \cdots \otimes 
v_{j_{i-1}}r_iu_{j_i}\otimes v_{j_i}\otimes r_{i+1}\otimes \cdots \otimes r_p. 
\end{align*}
Then, it can be checked by direct computation (cf. Lemma 1.0.9 in \cite{Lod}) that
$h:=\sum_{i=0}^p (-1)^{i}h_i$ is the homotopy of the identity operator to $\psi \circ \phi$, i.e., 
\[  \left. \phantom{\frac{dy}{dx}} d \circ h+h \circ d \right\vert_{C_p^S(R,M)}=\id_{C_p^S(R,M)}-\psi \circ \phi. \]
Hence, the induced maps $\phi_\ast$ and $\psi_\ast$ are inverse to each other. 
\end{proof}

\end{document}